\documentclass[12pt]{amsart}
\usepackage{graphicx}
\newtheorem{theorem}{Theorem}[section]

\newtheorem{definition}[theorem]{Definition}
\newtheorem{remark}[theorem]{Remark}
\newtheorem{proposition}[theorem]{Proposition}
\newtheorem{lemma}[theorem]{Lemma}

\newtheorem{example}[theorem]{Example}
\newtheorem{examples}[theorem]{Examples}
\newtheorem{notation}[theorem]{Remark on notation}
\title[On the first time that an Ito process hits a barrier]{On the first time that an Ito process hits a barrier}
\author{Gerardo Hernandez-del-Valle}
\address{Statistics Department, Columbia University\\1255 Amsterdam Ave. Room 1005, New York, N.Y., 10027.}
\email{gerardo@stat.columbia.edu}
\keywords{Doob's $h$-transform, Brownian bridge, Bessel processes, potential theory, first hitting time}
\date{July 9, 2012}
\thanks{The research of the author was partially supported by Algorithmic Trading Management LLC}
\subjclass[2000]{Primary: 37A50, 60G07, 60H30}
\begin{document}
\maketitle
\begin{abstract} This work deals with first hitting time densities of Ito processes whose local drift can be modeled in terms of a solution to Burgers equation. In particular, we derive the densities of the first time that these processes reach a moving boundary.  We distinguish two cases: (a) the case in which the process has unbounded state space before absorption, and (b) the case in which the process has bounded state space before absorption.
The reason as to why this distinction has to be made will be clarified.

Next, we classify processes whose local drift can be expressed as a linear combination to solutions of Burgers equation. For example the local drift of a Bessel process of order 5 can be modeled as the sum of two solutions to Burgers equation and thus will be classified as of class $\mathcal{B}^2$. Alternatively, the Bessel process of order 3 has a local drift that can be modeled as a solution to Burgers equation and thus will be classified as of class $\mathcal{B}^1$.
Examples of diffusions within class $\mathcal{B}^1$, and hence those to which the results described within apply, are: Brownian motion with linear drfit, the 3D Bessel process, the 3D Bessel bridge, and the Brownian bridge.\end{abstract}
\section{Introduction}
This work deals with processes whose local drift is modeled in terms of a solution to Burgers equation
\begin{eqnarray*}
-\mu_t(t,x)=\frac{1}{2}\mu_{xx}(t,x)+\mu(t,x)\cdot\mu_x(t,x),\quad (t,x)\in\mathbb{R}^+\times\mathbb{R}.
\end{eqnarray*}
Some of these processes appear prominently in stochastic analysis. Its applications in finance, as well as its appearance  in statistical problems is also noteworthy.  Among examples of processes with this property we can mention: Brownian motion, Brownian motion with linear drift [for financial applications of this process see Section 5.5.8 in Karatzas and Shreve (1991)], the 3D Bessel process,  the 3D Bessel bridge  [some recent financial applications of this process can be found in Davis and Pistorius (2010), it is also related to the Kolmogorov-Smirnov and Cram\'er von Mises tests, see Gikhman (1957) and Kiefer (1959)], and  the Brownian bridge [for applications of this process in tests of Kolmogorov-Smirnov type see Kolmogorov (1933), Smirnov (1948), and Andel (1967)] . 

In fact any Bessel process of odd order greater than or equal to 3 has a local drift which can be modeled as linear combinations of solutions to Burgers equation. For instance, the Bessel process of order 5 has a local drift which is a sum of two solutions of Burgers equation. In general a Bessel process of order $2k+1$, $k\geq 1$, has a local drift which can be written as a sum of $k$ solutions to Burgers equation. This leads us to classify SDEs to be of class $\mathcal{B}^k$ if its local drift is a sum of $k$ solutions to Burgers equation. Regarding the related topic of hitting times of Bessel processes there are a number of recent papers: e.g.  Salminen and Yor (2011) or  Alili and Patie (2010).


In this paper we derive the densities of the first time that processes of class $\mathcal{B}^1$ reach moving boundaries. Furthermore  we assume that the moving boundaries are real valued  and twice continuously differentiable. To this end we distinguish two cases: (a) the case in which the process has unbounded state space before absorption, and (b) the case in which the process has bounded state space before absorption. 
An example of the second case is the density of the first time that a 3-D Bessel bridge started at $y>0$, and absorbed at zero at time $s$, hits a fixed level $a$, where $y<a$. That is, the 3-D Bessel bridge lives on $(0,a)$ before being absorbed at either level $a$, or at level $0$ at time $s$, see for instance Hernandez-del-Valle (2012).

The paper is organized as follows in Section \ref{sec1} an $h$-transform and general notation are introduced. Next, in Section \ref{sec2}, the unbounded state space case is discussed. Section \ref{sec3}, deals with problems where the process lives on a bounded interval. Next, in Section \ref{secbur}, we give a more detailed description of class $\mathcal{B}^1$ and its relationship with class $\mathcal{B}^2$.We conclude in Section \ref{sec4} with some final comments and remarks. 
\section{Preliminaries}\label{sec1}
\begin{notation}  

As in the analysis of diffusion processes, PDEs with derivatives with respect to $(t,x)$ are called backward equations, whereas PDEs with derivatives in $(s,y)$ are called forward equations.

Furthermore, through out this work  (i) $B=\{B_t,\mathcal{F}_t\}_{t\geq 0}$ stands for one-dimensional standard Brownian motion. (ii) For a given function, say $w$, partial differentiation with respect to a given variable, say $x$, will be denoted as $w_x$.
\end{notation}

In the following sections our main tools will be:  (1) An $h$-transform, see Theorem \ref{thm1}, and (2) Ito's lemma, see Lemma \ref{lema}. [Regarding $h$-transforms the reader may consult
 Doob (1984) or Pinsky (1995).]

\begin{theorem}\label{thm1} Let $h$ be of class $C^{1,2}(\mathbb{R}^+\times\mathbb{R})$ as well as a solution to the backward heat equation 
\begin{eqnarray}\label{heatp}
-h_t=\frac{1}{2} h_{xx}.
\end{eqnarray}
 Furthermore, consider processes $X$, and $Y$ which respectively satisfy (at least in the weak sense), the following equations (each under their corresponding measures $\mathbb{P}$, and $\mathbb{Q}$),
\begin{eqnarray}\label{sde1}
(\mathbb{P})\quad dX_t&=&\frac{h_x(t,X_t)}{h(t,X_t)}dt+dB_t\\
\label{sde2}
(\mathbb{Q})\quad dY_t&=&dB_t.
\end{eqnarray}
Moreover, suppose that $f$ is real valued and integrable. Then the following identity holds
\begin{eqnarray}\label{change}
\mathbb{E}_{t,x}^{\mathbb{P}}[f(X_\tau)]=\mathbb{E}^{\mathbb{Q}}_{t,x}\left[\frac{h(\tau,Y_\tau)}{h(t,x)}f(Y_\tau)\right].
\end{eqnarray}
\end{theorem}
\begin{proof} For the proof, see Theorem 2.10 in Hernandez-del-Valle (2011). 
\end{proof}
We recall Ito's lemma in the case in which the drift is a deterministic function of time.
\begin{lemma}\label{lema} Let $f(\cdot)$ be a real-valued differentiable function, $h$ a solution of the one-dimensional backward heat equation,  and processes $Z$  and $S$ have the following dynamics
\begin{eqnarray*}
dZ_t&=&f'(t)dt+dB_t\\
dS_t&=&-f'(t)S_tdB_t
\end{eqnarray*}
for $0\leq t<\infty$, under some meaure $\mathbb{Q}$. Then
\begin{eqnarray*}
S_\cdot\cdot h(\cdot, Z_\cdot)
\end{eqnarray*}
is a $\mathbb{Q}$-martingale.
\end{lemma}
\begin{proof} From Ito's lemma 
\begin{eqnarray*}
dh(t,Z_t)&=&h_t(t,Z_t)dt+h_z(t,Z_t)f'(t)dt\\
&&+h_z(t,Z_t)dY_t+\frac{1}{2}h_{zz}(t,Z_t)dt\\
&=&h_z(t,Z_t)f'(t)dt+h_z(t,Z_t)dY_t.
\end{eqnarray*}
Hence,
\begin{eqnarray*}
d\left[S_t\cdot h(t,Z_t)\right]&=&f(t,Z_t)dS_t+S_tdh(t,Z_t)\\
&&+dh(t,Z_t)\cdot dS_t\\
&=&h(t,Z_t)\left[-f'(t)S_tdY_t\right]\\
&&+S_th_z(t,Z_t)f'(t)dt+S_th_z(t,Z_t)dY_t\\
&&-h_z(t,Z_t)S_tf'(t)dt\\
&=&\left[S_th_z(t,Z_t)-f'(t)X_th(t,Z_t)\right]dY_t.
\end{eqnarray*}
\end{proof}

\section{Unbounded state space}\label{sec2}
\begin{remark}\label{rp1}
Henceforth let $h$ be a solution to the one-dimensional backward heat equation (\ref{heatp}). Let  process $X$ have the following dynamics
\begin{eqnarray}\label{process}
dX_t=\frac{h_x(t,X_t)}{h(t,X_t)}dt+dB_t,\qquad X_0= y\in\left\{\begin{array}{ll}\mathbb{R}& \hbox{or}\\ \mathbb{R}^+\end{array}\right. ,
 \end{eqnarray}
for $0\leq t<\infty$. And assume the drift $h_x/h$ satisfies the Ito conditions.
 \end{remark}
 
 In this section, we specialize to the case in which the process $X$, with dynamics as in (\ref{process}), has unbounded state space before absorption.  In particular we find the density of the first time that $X$ hits a real-valued and twice continuously differentiable function $f$. To this end let us first define the following stopping times.
 \begin{definition}
 Given the constant $a\in\mathbb{R}$ and the real-valued, twice continuously  differentiable function $f(\cdot)$---which we refer to as a ``mo\-ving boun\-dary''---we define the following {\it stopping times\/}
\begin{eqnarray}\label{stop}
T&:=&\inf\left\{t\geq 0|X_t=a+\int_0^tf'(u)du\right\}\qquad y\not =a\\
\nonumber T^B&:=&\inf\left\{t\geq 0|B_t=a+\int_0^tf'(u)du\right\}.
\end{eqnarray}
Furthermore, let  $p^B_f(\cdot)$ be the density of $T^B$, that is, the density of the first time that a one-dimensional Wiener process reaches a deterministic moving boundary $f$. For detailed historical and technical account (making use of integral equations) of this problem see Peskir (2001).
\end{definition}
  
Next, we present the main result of this section.
 \begin{theorem}\label{thm2} Suppose that $X$ has dynamics as in (\ref{process}), $T$ is as in (\ref{stop}), and $a\not=y$. Then 
\begin{eqnarray*}
\mathbb{P}_y(T\in du)=\frac{h\left(u,a+\int_0^uf'(v)dv\right)}{h(0,y)}p^B_f(u)du,
\end{eqnarray*}
for $u\geq 0$.
\end{theorem}
\begin{proof} For $T$  as in (\ref{stop}), it follows from Theorem \ref{thm1} that
\begin{eqnarray*}
\mathbb{P}_y(T<t)&=&\mathbb{E}_y\left[\mathbb{I}_{(T<t)}\right]\\
&=&\mathbb{E}_y^{\mathbb{Q}}\left[\frac{h(t,Y_t)}{h(0,Y_0)}\mathbb{I}_{(T<t)}\right].
\end{eqnarray*} 
From Girsanov's theorem, and given that $\tilde{Y}$ is a $\tilde{\mathbb{Q}}$-Wiener process we have
\begin{eqnarray*}
\phantom{xxxxx}&=&\mathbb{E}^{\tilde{\mathbb{Q}}}_y\left[e^{-\int_0^tf'(u)d\tilde{Y}_u-\frac{1}{2}\int_0^u(f'(u))^2du}\frac{h(t,\tilde{Y}_t+\int_0^tf'(u)du)}{h(0,y)}\mathbb{I}_{(T<t)}\right].
\end{eqnarray*}
Finally, from Lemma \ref{lema} and the optional sampling theorem
\begin{eqnarray*}
\phantom{xxxxx}&=&\mathbb{E}^{\mathbb{Q}}_y\left[e^{-\int_0^Tf'(u)d\tilde{Y}_u-\frac{1}{2}\int_0^T(f'(u))^2du}\frac{h(T,a+\int_0^Tf'(u)du)}{h(0,y)}\mathbb{I}_{(T<t)}\right]\\
&=&\int_0^t \frac{h\left(u,a+\int_0^uf'(v)dv\right)}{h(0,y)}p^B_f(u)du.
\end{eqnarray*}
\end{proof}
Examples of processes which satisfy equation (\ref{process}) and that have unbounded domain before hitting the boundary $f$ are:
  \begin{examples}\label{exj} 
\begin{enumerate}
\item[(i)] Standard Brownian motion, where $h(t,x)=c$.
\item[(ii)] Brownian motion with linear drift, where
\begin{eqnarray*}
h(t,x)=e^{\pm \lambda x+\frac{1}{2}\lambda(s-t)},\qquad (t,x)\in\mathbb{R}^+\times \mathbb{R}.
\end{eqnarray*}
\item[(iii)] Brownian bridge, where
\begin{eqnarray*}
h(t,x)=\frac{1}{\sqrt{2\pi(s-t)}}e^{-\frac{x^2}{2(s-t)}},\qquad(t,x)\in\mathbb[0,s]\times\mathbb{R}.
\end{eqnarray*}
\item[(iv)] 3D Bessel process, where $h(t,x)=x$.
\item[(v)] 3D Bessel bridge, where
\begin{eqnarray*}
h(t,x)&=&\frac{x}{\sqrt{2\pi(s-t)^3}}e^{-\frac{x}{2(s-t)}},\qquad (t,x)\in[0,s]\times\mathbb{R}^+.
\end{eqnarray*}
\end{enumerate}
\end{examples}

\begin{example}\label{exa1} In Figures \ref{fig1} and \ref{fig2} the theoretical density and distribution of the first time that a Brownian bridge started at $y=1$, absorbed at time $s=3$ at level $c=0$ hits the linear barrier
\begin{eqnarray*}
f(t)= 2-t,\qquad t\geq 0
\end{eqnarray*}
is compared with $n=5500$ simulations. See Durbin and Williams (1992).  
\end{example}
\begin{example}\label{exa3} In Figures \ref{fig1} and \ref{fig2}, the theoretical densities and distributions of the first time that a 3-D  Bessel process started at $y=3$ and absorbed at time $s=4$, reaches level $a=1$. For a general overview of the 3D Bessel bridge, see Revuz and Yor (2005).  The case in which a level is reached from below is in general studied in Pitman and Yor (1999) or Hernandez-del-Valle (2012).
\end{example}
\section{Bounded domain}\label{sec3}
In this section we will also also use the following:\\[0.3cm]
\indent Given that $B$ is a one-dimensional Wiener process started at $y$ and 
\begin{eqnarray*}
T_0&:=&\inf\left\{t\geq 0|B_t=0\right\}\\
T&:=&\inf\left\{t\geq 0|B_t=a\right\},\quad 0<y<a.
\end{eqnarray*}
We recall
\begin{eqnarray*}
\mathbb{P}_y(T\wedge T_0\in dt)&:=&\frac{1}{\sqrt{2\pi t^3}}\sum\limits_{n=-\infty}^\infty\Bigg{[}(2na+y)\exp\left\{-\frac{(2na+y)^2}{2t}\right\}\\
&&+(2na+a-y)\exp\left\{-\frac{(2na+a-y)^2}{2t}\right\}\Bigg{]}dt,\\
\nonumber\mathbb{P}_{t,y}(T\in s,T_0>s)
&:=&\frac{1}{\sqrt{2\pi (s-t)^3}}\sum\limits_{n=-\infty}^\infty\Bigg{[}(2na+a-y)\\
&&\qquad\qquad\qquad\times\exp\left\{-\frac{(2na+a-y)^2}{2(s-t)}\right\}\Bigg{]},\\
\mathbb{P}_y(T_0\in dt)&:=&\frac{y}{\sqrt{2\pi t^3}}\exp\left\{-\frac{y^2}{2t}\right\}dt.
\end{eqnarray*} 
See for instance Chapter 2, Section 8 in Karatzas and Shreve (1991).\\[0.3cm]
\indent In this section, we specialize to the case in which the process $X$, with dynamics as in (\ref{process}), has bounded state space before absorption. The main result of this section is the following.
\begin{theorem}
Given that $X$ has dynamics as in (\ref{process}) and $T$ is as in (\ref{stop}). We have for $0\leq t\leq s$ and $0<y\leq a$ that \begin{eqnarray*}
\mathbb{P}_y(T\in du)=\frac{h\left(u,a\right)}{h(0,y)}\left[\mathbb{P}_y(T\wedge T_0\in du)-\mathbb{P}_y(T_0\in du)\right],
\end{eqnarray*}
where $0\leq u\leq s$.
\end{theorem}
\begin{proof}  We follow the proof of Theorem \ref{thm2}. However, we must take into account the fact that the $\mathbb{Q}$-Wiener process $Y$ is absorbed at zero. Given that $T$ is as in (\ref{stop})
\begin{eqnarray*}
\mathbb{P}_y(T<t)&=&\mathbb{E}_y\left[\mathbb{I}_{(T<t)}\mathbb{I}_{(T_0>t)}\right]\\
&=&\mathbb{E}^{\mathbb{Q}}_y\left[\frac{h(t,Y_t)}{h(0,Y_0)}\mathbb{I}_{(T<t,T_0>t)}\right]\\
&=&\mathbb{E}^{\mathbb{Q}}_y\left[\frac{h(T,a)}{h(0,Y_0)}\mathbb{I}_{(T<t,T_0>t)}\right],\\
\end{eqnarray*}
where the last line follows from the optional sampling theorem. Finally recall the identity
\begin{eqnarray*}
\mathbb{P}_y(T<t,T_0>t)&=&\mathbb{P}_y(T_0>t)-\mathbb{P}_y(T>t,T_0>t)\\
&=&\mathbb{P}_y(T_0>t)-\mathbb{P}_y(T\wedge T_0>t)\\
&=&\mathbb{P}_y(T\wedge T_0<t)-\mathbb{P}_y(T_0<t).
\end{eqnarray*}
\end{proof}
\begin{examples}\label{exx1} Some examples are: (i) 3D Bessel process (reaching a fixed level from below) (ii) 3D Bessel bridge (reaching a fixed level from below). Regarding the first hitting probabilities of general Bessel processes see Wendel (1980), 
or Betz and Gzyl (1994a, 1994b).
\end{examples}
\begin{example}\label{exa8} In Figure \ref{fig3} the theoretical distribution of the first time that a 3D Bessel process reaches $a=1.5$ from below, is compared with a simulation $n=5500$ (see hard line).
\end{example} 
\begin{example}
If we set
\begin{eqnarray*}
h^a(s-t,x):=\mathbb{P}_{t,x}(T\in s,T_0>s),
\end{eqnarray*}
and define a process $\tilde{Y}$ to be as in
\begin{eqnarray*}
d\tilde{Y}_t&=&\frac{h^a_y(s-t,\tilde{Y}_t)}{h^a(s-t,\tilde{Y}_t)}dt+dB_t,\quad 0<t<s\\
\nonumber\tilde{Y}_s&=&a.
\end{eqnarray*}
One can show that this process has state space $(0,a)$ for  $t\in[0,s)$ [see Hernandez-del-Valle (2011)]. In Figure \ref{fig8} the density and distribution of the Wiener process, started at $x=1/2$, absorbed at zero and that reaches level $a=2$ for the first time at $s=2$ is plotted at $t=1$ and $t=7/8$.
\end{example}
\section{Heat polynomials and Burgers equation}\label{secbur}
In this section we provide a characterization of the results described in Sections \ref{sec2} and \ref{sec3}.

To this end let us first introduce the following classification of SDEs.
\begin{definition} We will say that process $X$, which satisfies the following equation
\begin{eqnarray*}
dX_t=\mu(t,X_t)dt+dB_t,
\end{eqnarray*}
is of class $\mathcal{B}^n$, $n=1,2,\dots$, if its local drift $\mu$ can be expressed as
\begin{eqnarray}\label{burger}
\mu(t,x)=\sum\limits_{j=1}^n\frac{h^j_x(t,x)}{h^j(t,x)}.
\end{eqnarray}
Where each $h^j$ is a solution to the backward heat equation (\ref{heatp}).
\end{definition}
Making use of this classification it follows that:
\begin{remark}  Processes $X_1$, $X_2$, and $X_3$ which respectively satisfy the following equations
\begin{eqnarray}\label{prx}
\left\{\begin{array}{ll} (\mathbb{P}^1) & dX_1(t)=\frac{1}{X_1(t)}dt+dB_t\\
(\mathbb{P}^2) & dX_2(t)=-\frac{X_2(t)}{s-t}dt+dB_t\\
(\mathbb{P}^3) & dX_3(t)=\left[\frac{1}{X_3(t)}-\frac{X_3(t)}{s-t}\right]dt+dB_t
\end{array}\right.
\end{eqnarray}
are of class $\mathcal{B}^1$. That is, the 3-D Bessel process $X_1$, the Brownian bridge $X_2$, and the 3-D Bessel bridge have a local drift which is a solution to Burgers equation. This statement is verified by using the following solutions to the heat equation correspondingly
\begin{eqnarray}\label{drifts}
\left\{\begin{array}{l} k(t,x)=x,\quad g(t,x)=\frac{1}{\sqrt{2\pi(s-t)}}\exp\left\{-\frac{x^2}{2(s-t)}\right\}\\
h(t,x)=\frac{x}{\sqrt{2\pi(s-t)^3}}\exp\left\{-\frac{x^2}{2(s-t)}\right\},
\end{array}\right.
\end{eqnarray}
together with the Cole-Hopf transform which relates Burgers equation with the heat equation.
\end{remark}
Examples of processes which are not $\mathcal{B}^1$ are the following.
\begin{example} The Bessel process of odd order $2n+1$, $n=1,2,\dots$ is of class $\mathcal{B}^n$.

Recall that the Bessel process of order $m\in\mathbb{N}$ is the solution to 
\begin{eqnarray*}
dX_t=\frac{m-1}{2X_t}dt+dB_t,
\end{eqnarray*}
if $m=2n+1$
\begin{eqnarray*}
dX_t&=&\frac{(2n+1)-1}{2X_t}dt+dB_t\\
&=&\frac{n}{X_t}dt+dB_t\\
&=&\left[\frac{k_x(t,X_t)}{k(t,X_t)}+\cdots+\frac{k_x(t,X_t)}{k(t,X_t)}\right]dt+dB_t,
\end{eqnarray*}
where $k$ is as in (\ref{drifts}).
\end{example}
However there exists a least one important  process which is both $\mathcal{B}^1$ and $\mathcal{B}^2$. 
\begin{proposition} The 3-D Bessel bridge process $X_3$, which solves (\ref{prx}.$\mathbb{P}^3$) is $\mathcal{B}^1$ and $\mathcal{B}^2$.
\end{proposition}
\begin{proof} it follows by verifying that for $k$, $g$, and $h$ as in (\ref{drifts}) the following identity holds
\begin{eqnarray*}
\frac{h_x}{h}=\frac{k_x}{k}+\frac{g_x}{g}.
\end{eqnarray*}
\end{proof}

In turn, this feature of process $X_3$ [{\it i.e.\/} a process that is both $\mathcal{B}^1$ and $\mathcal{B}^2$], leads to some interesting properties. [See Hernandez-del-Valle (2011).] The next natural question to ask is if there are more processes with such property.

In this direction we will show that process $X_3$ is the only element of both $\mathcal{B}^1$ and $\mathcal{B}^2$. In the case in which the generating functions $w$, of class $\mathcal{B}^1$, are heat polynomials. [That is if $w$ is a solution to the heat equation and $w_x/w$ models the local drift of a process $X$]

\begin{definition}\label{defineheat} (Heat polynomials) [Widder and Rosenbloom (1959)]. A heat polynomial $v_j(x,t)$ of degree $j$ is defined as the coefficient of $z^n/n!$ in the power series expansion
\begin{eqnarray}\label{wheat}
e^{xz+\frac{1}{2}z^2t}=\sum\limits_{n=0}^\infty v_n(x,t)\frac{z^n}{n!}.
\end{eqnarray}
An associated function $w_n(x,t)$ is defined as
\begin{eqnarray}\label{wdheat}
w_n(x,t)=g(t,x)v_n\left(x,-t\right)(t/2)^{-n},
\end{eqnarray}
where $g$, as in (\ref{drifts}), is the fundamental solution to the heat equation.
\end{definition}
\begin{remark} Observe that if processes $X_1$, $X_2$, and $X_3$ are as in (\ref{prx}) then their local drifts can be described in terms of heat polynomials, Definition \ref{defineheat}. In the case of $X_1$ its corresponding polynomial is $v_1$. Alternatively for $X_2$ and $X_3$ their corresponding polynomials are $w_1$ and $w_2$ respectively.
\end{remark}
The main result of this section is the following.
\begin{theorem} If process $X\in\mathcal{B}^1$. [That is $X$ is  a solution to
\begin{eqnarray*}
dX_t=\frac{h_x(t,X_t)}{h(t,X_t)}dt+dB_t
\end{eqnarray*}
and $h$ solves the backward heat equation $-h_t=\frac{1}{2}h_{xx}$.] And $h$ is either a heat (\ref{wheat}) or derived heat (\ref{wdheat}) polynomial. Then the only process which is also $\mathcal{B}^2$ is the 3-D Bessel bridge $X_3$, which has  dynamics as in (\ref{prx}.$\mathbb{P}^3$).
\end{theorem}
\begin{proof} Given that $v_n$ and $w_n$ are as in (\ref{wheat}) and (\ref{wdheat}) respectively; and letting $w'$ stand for differentiation with respect to the first variable we have that
\begin{eqnarray*}
\frac{w'_0(x,t)}{w_0(x,t)}&=&-\frac{x}{t}\quad\hbox{and}\quad\frac{v'_1(x,t)}{v_1(x,t)}=\frac{1}{x}.
\end{eqnarray*}
Next, since [see p. 225 in Rosenbloom and Widder (1959)]
\begin{eqnarray*}
w'_{n-1}(x,t)=-\frac{1}{2}w_n(x,t)
\end{eqnarray*}
it follows that
\begin{eqnarray}\label{hj}
\nonumber\frac{w'_n(x,t)}{w_n(x,t)}&=&-\frac{x}{t}+\frac{2n}{t}\frac{w_{n-1}(x,t)}{w_n(x,t)}\\
\nonumber&=&-\frac{x}{t}-\frac{n}{t}\left[\frac{w_{n-1}(x,t)}{w'_{n-1}(x,t)}\right]\\
&=&\frac{w'_0(x,t)}{w_0(x,t)}-\frac{n}{t}\left[\frac{w_{n-1}(x,t)}{w'_{n-1}(x,t)}\right].
\end{eqnarray}
Alternatively from (\ref{wdheat})
\begin{eqnarray*}
\frac{w_{n-1}}{w'_{n-1}}&=&\frac{g(x,t)v_{n-1}(x,-t)(t/2)^{-n+1}}{-\frac{1}{2}w_n(x,t)}\\
&=&\frac{g(x,t)v_{n-1}(x,-t)(t/2)^{-n+1}}{-\frac{1}{2}g(x,t)v_n(x,-t)(t/2)^{-n}}\\
&=&-\frac{tv_{n-1}(x,-t)}{v_n(x,-t)}.
\end{eqnarray*}
This implies, from (\ref{hj}),  that
\begin{eqnarray}\label{hi}
\nonumber\frac{w'_n(x,t)}{w_n(x,t)}
\nonumber&=&\frac{w'_0(x,t)}{w_0(x,t)}-\frac{n}{t}\left[\frac{w_{n-1}(x,t)}{w'_{n-1}(x,t)}\right]\\
\nonumber&=&\frac{w'_0(x,t)}{w_0(x,t)}-\frac{n}{t}\left[-\frac{tv_{n-1}(x,-t)}{v_n(x,-t)}\right]\\
&=&\frac{w'_0(x,t)}{w_0(x,t)}+n\frac{v_{n-1}(x,-t)}{v_n(x,-t)}.
\end{eqnarray}
However, since
\begin{eqnarray*}
v'_n(x,t)=nv_{n-1}(x,t)
\end{eqnarray*}
[see equation (1.9) in Widder and Rosenbloom (1959)] we have, from (\ref{hi}), that
\begin{eqnarray*}
\nonumber\frac{w'_n(x,t)}{w_n(x,t)}
&=&\frac{w'_0(x,t)}{w_0(x,t)}+\frac{nv_{n-1}(x,-t)}{v_n(x,-t)}\\
&=&\frac{w'_0(x,t)}{w_0(x,t)}+\frac{v'_{n}(x,-t)}{v_n(x,-t)}.
\end{eqnarray*}
In general, since $v_n(x,t)$ is a solution to the backward heat equation, then $v_n(x,-t)$ is a solution to the forward equation. This is true as long as $n>1$. However if $n=1$, $v_1(x,-t)$ is also a solution to the backward equation becuase it does not depend on $t$. In this case we have that
\begin{eqnarray*}
\frac{w'_1(x,t)}{w_1(x,t)}&=&\frac{x}{t}+\frac{1}{x},
\end{eqnarray*}
as claimed.
\end{proof}



\section{Concluding remarks}\label{sec4}
In this work we study processes $X$, which have  local drift modeled in terms of solutions to Burgers equation. In particular, we find the density of the first time that such processes hit a moving boundary. Next, we propose a classification of SDEs in terms of solutions of Burgers equation. We say that process $X$ is of class $\mathcal{B}^j$ if its local drift can be expressed as a sum of $j$ solutions to Burgers equation.

We note that the 3-D Bessel process, the 3-D Bessel bridge, and the Brownian bridge are all $\mathcal{B}^1$. However we show that the 3-D Bessel bridge. Furthermore, we show, that it is the only process which satisfies this property if the solutions of Burgers equation is constructed by use of heat polynomials. A more detailed study of this classification is work in progress.

\begin{figure}
\hspace{-1.3 cm}\includegraphics[scale=.5,height=14cm,width=14cm]{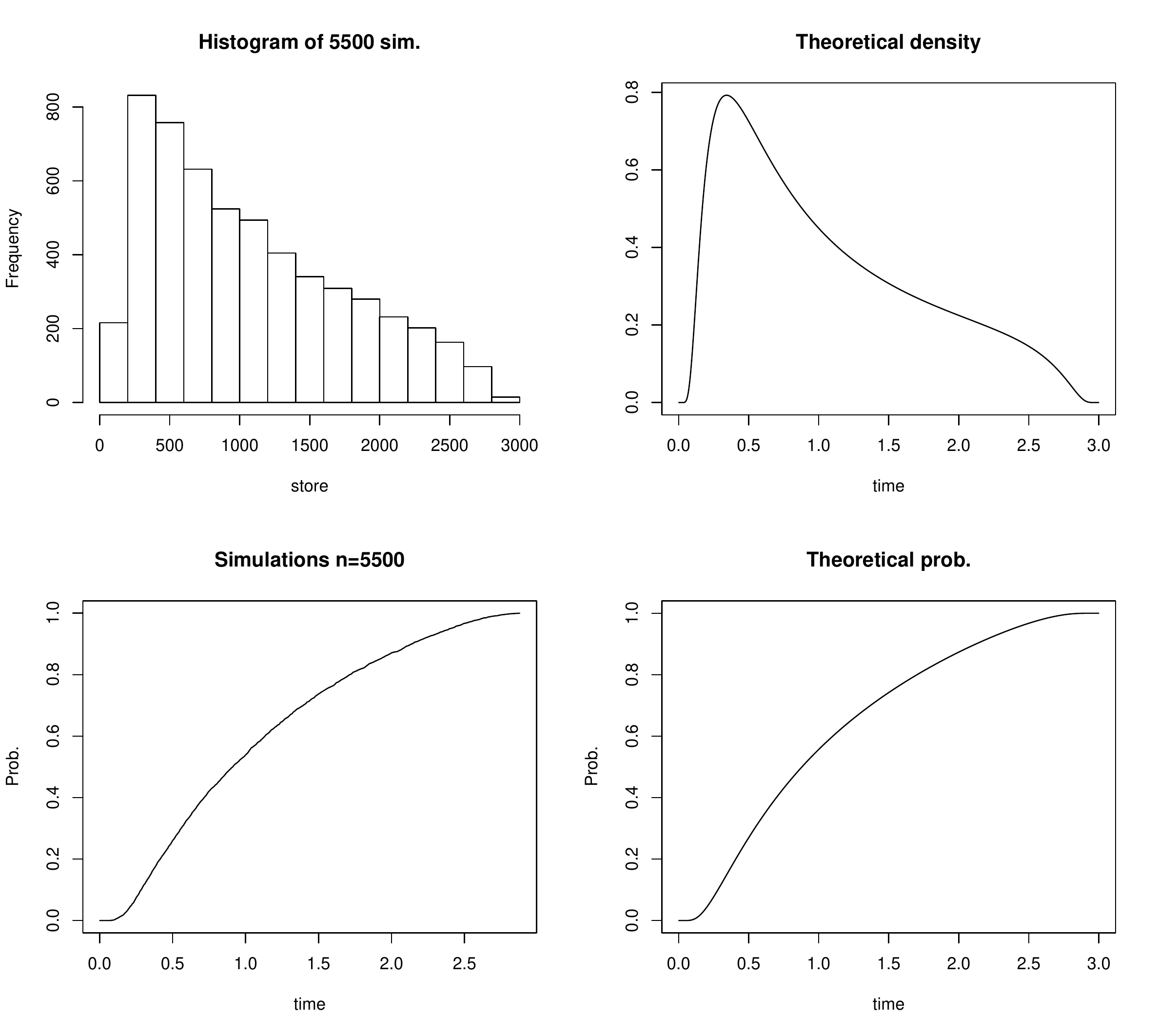}
\caption{(Example \ref{exa1}). The graph is plotted in R. The upper left graph is the histogram of  the (simulated) first time that a Brownian bridge started at $y=1$, and absorbed at $c=0$ at time $s=3$, reaches the linear boundary $f(t)=2-t$. In the upper right frame we have  its theoretical density. In the lower left we have the simulated distribution, and finally on its right we have its theoretical counterpart. }\label{fig1}
\end{figure}
\begin{figure}
\hspace{-1.3 cm}\includegraphics[scale=.5,height=14cm,width=14cm]{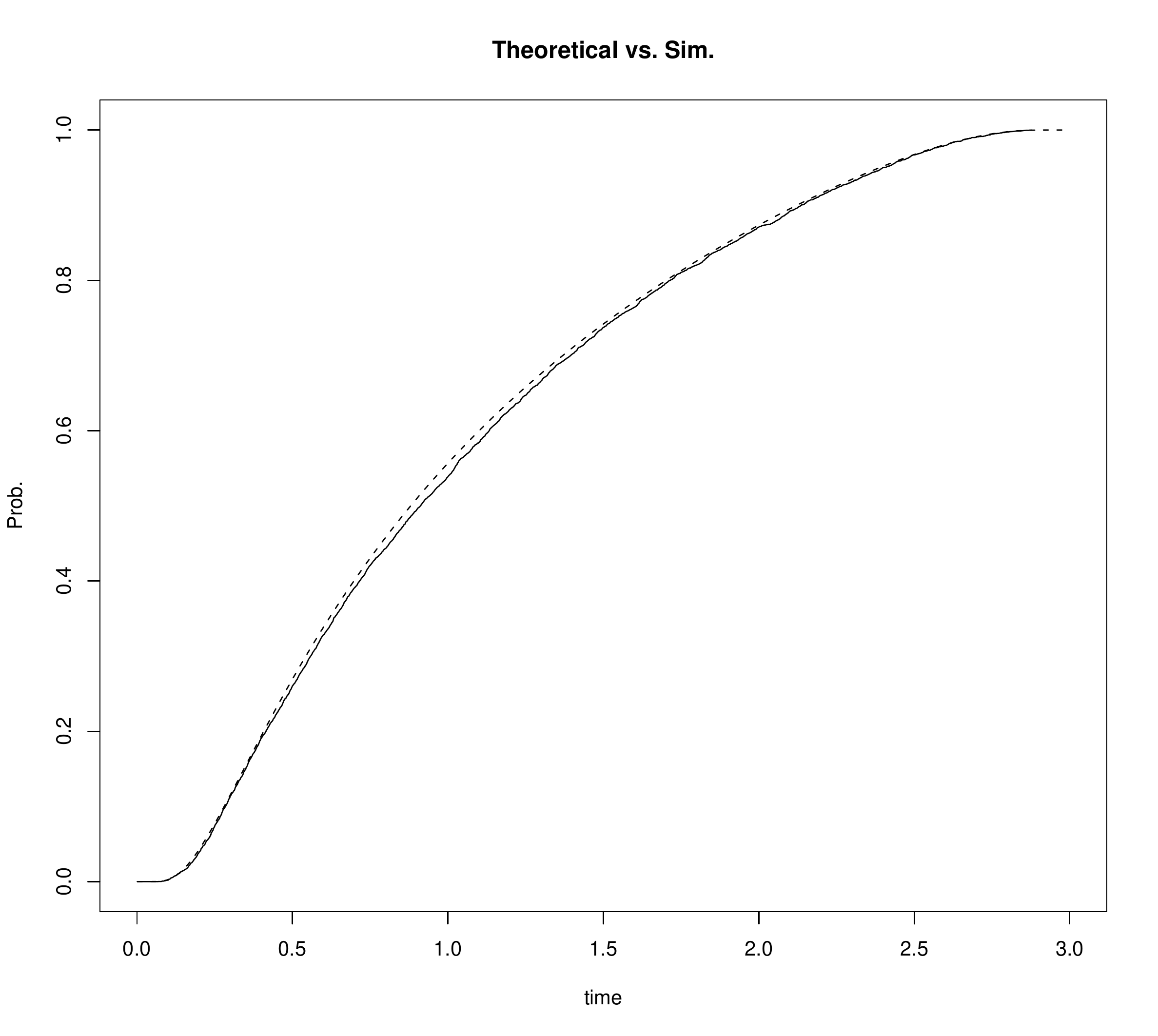}
\caption{(Example \ref{exa1}). The graph is plotted in R. The dotted line is the theoretical probability. The hard line is a simulation with $n=5500$. }\label{fig2}
\end{figure}

\begin{figure}
\hspace{-1.3 cm}\includegraphics[scale=.5,height=14cm,width=14cm]{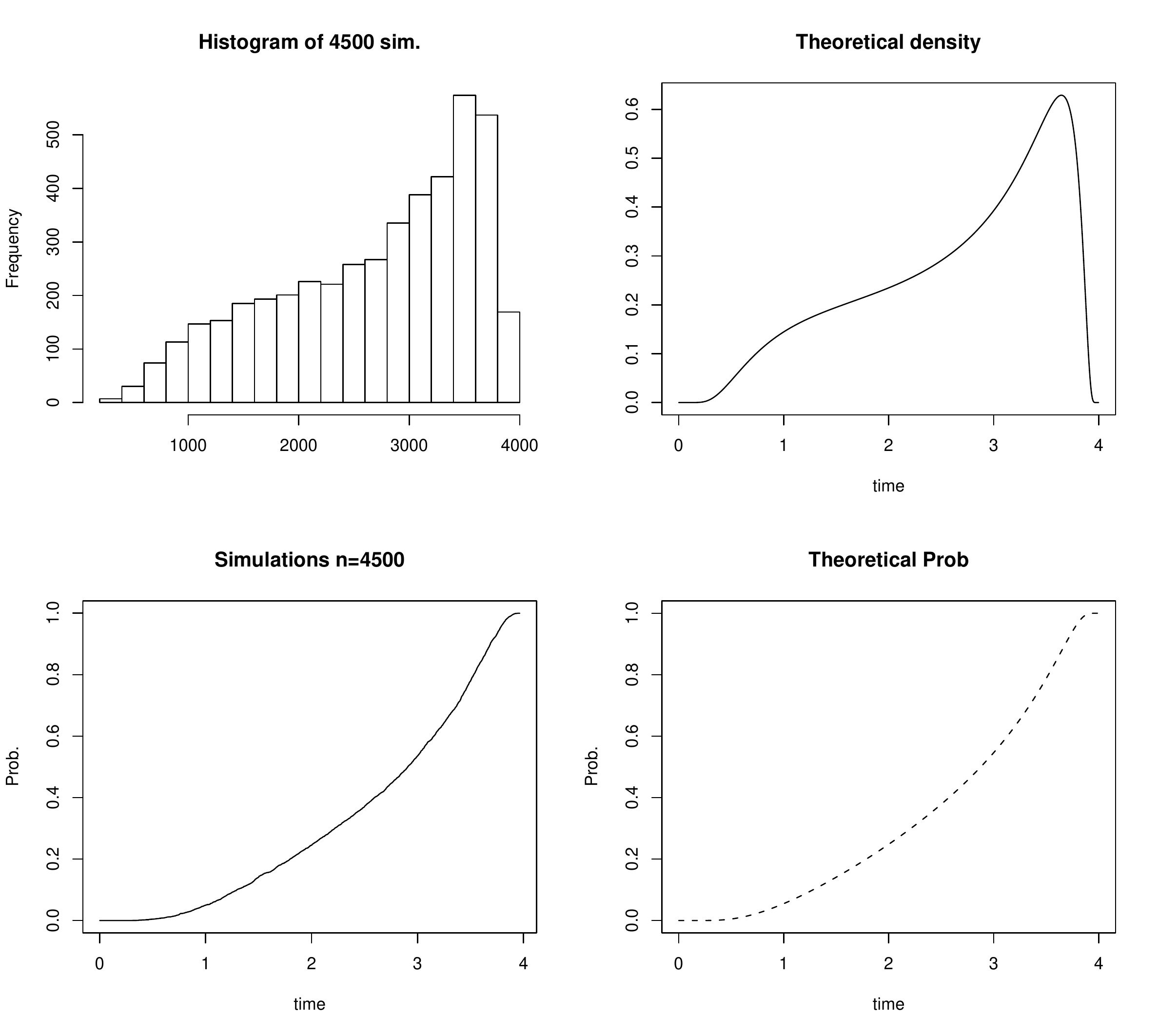}
\caption{(Example \ref{exa3}). The graph is plotted in R. The upper left graph is the histogram of  the (simulated) first time that a 3-D Bessel bridge started at $y=3$, and absorbed at $s=4$, reaches level $a=1$. In the upper right frame we have the its theoretical density. In the lower left we have the simulated distribution, and finally on its right we have its theoretical counterpart. }\label{fig4}
\end{figure}

\begin{figure}
\hspace{-1.3 cm}\includegraphics[scale=.5,height=14cm,width=14cm]{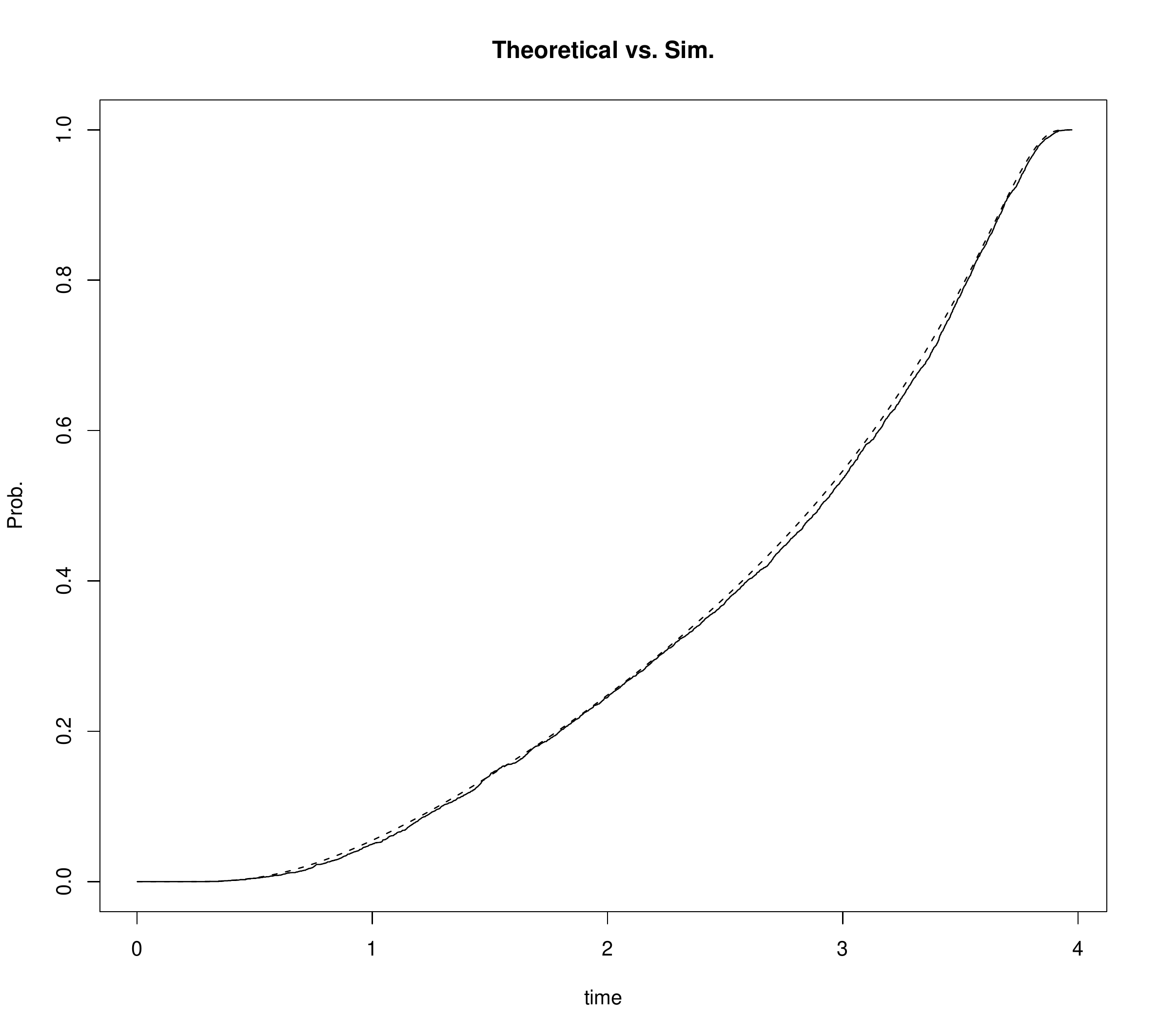}
\caption{(Example \ref{exa3}). The graph is plotted in R. The dotted line is the theoretical probability. The hard line is a simulation with $n=4500$. }\label{fig5}
\end{figure}

\begin{figure}
\hspace{-0.9 cm}\includegraphics[scale=.5,height=12cm,width=12cm]{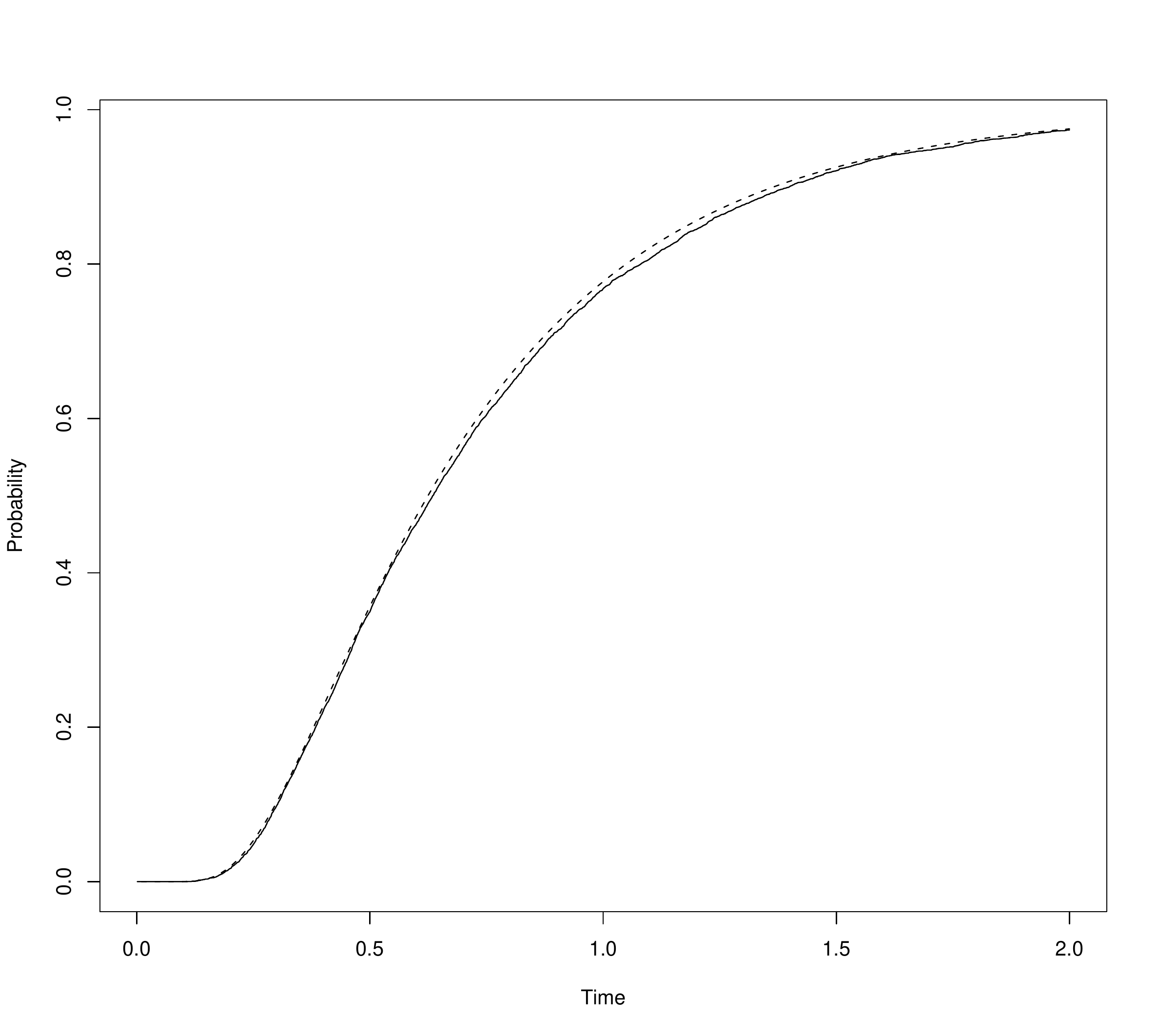}
\caption{(Example \ref{exa3}). The graph is plotted in R. The dotted line is the theoretical probability with $a=1.5$. The hard line is a simulation with $n=5500$. }\label{fig3}
\end{figure}

\begin{figure}
\hspace{-.5 cm}\includegraphics[scale=.5,height=13cm,width=13cm]{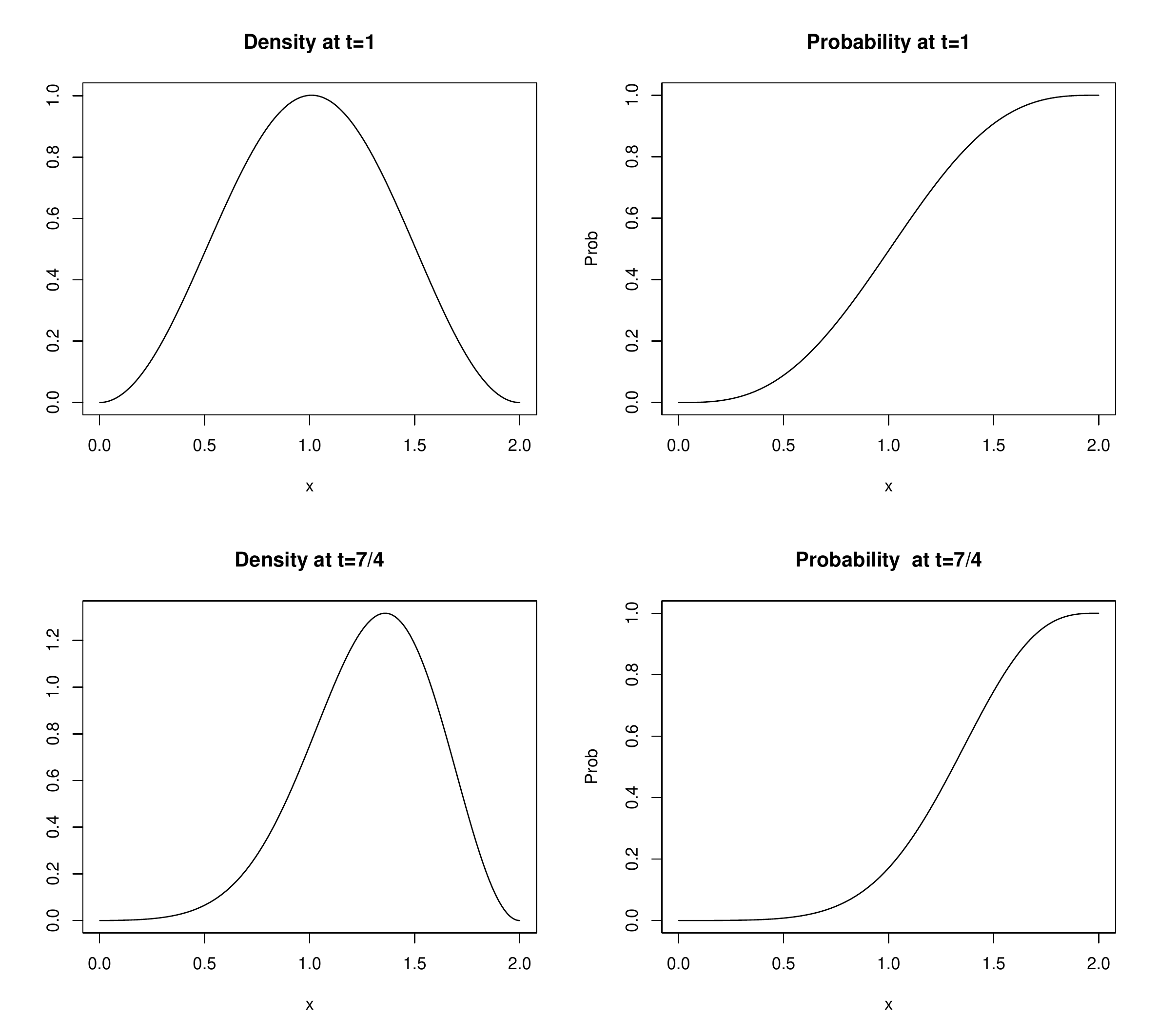}
\caption{(Example \ref{exa8}). The graphs are plotted in R. The upper left graph is the density of a Wiener process (started at $x=1/2$), absorbed at zero and that reaches $a=2$ for the first time at $s=2$; evaluated at $t=1$. Th upper right graph is the corresponding distribution at $t=1$. In the lower left figure we have the density at time $t=7/4$. Finally, the lower right panel is its corresponding distribution.}   \label{fig8} 
\end{figure}


\begin{thebibliography}{xx}
\bibitem{ap} Alili, L. and P. Patie (2010) Boundary--crossing identities for diffusions having the time--inversion property. {\it J. Theor. Probab.\/}, {\bf 23}, No. 1, pp. 65--84.

\bibitem{an} Andel J. (1967)  Local asymptotic power and efficiency of tests of Kolmogorov-Smirnov type, {\it Ann. Math. Stat.\/}, {\bf 38} No. 6, pp. 1705--1725.
\bibitem{beg1} Betz, C. and H. Gzyl. (1994a).  Hitting spheres from the exterior. {\it Ann. Probab.\/}, {\bf 22}, pp. 177--179.
\bibitem{beg2} Betz, C. and H. Gzyl. (1994b). Hitting spheres with Brownian motion and Sommerfeld's radiation condition. {\it J. Math. Anal. Appl.\/}, {\bf 182}, pp. 301--308.
\bibitem{dp} Davis, M. H. A. and M. R. Pistorius (2010) Quantification of counterparty risk via Bessel bridges. Available at SSRN: http://ssrn.com/abstract=1722604.
\bibitem{do} Doob, J. L.  (1984) Classical Potential Theory and Its Probabilistic Counterpart, Springer-Verlag, New York.
\bibitem{du} Durbin,  J. and D. Williams (1992) The first-passage density of the Brownian motion process to a curved boundary. {\it J. Appl. Probab.\/} {\bf 29} No. 2 pp. 291--304.
\bibitem{gg} Gikhman, I. I. (1957) On a nonparametric criterion of homogeneity for $k$ samples. {\it Theory Probab. Appl.\/},
{\bf 2}, pp. 369--373.
\bibitem{hh} Hernandez-del-Valle, G. (2011) On changes of measure and  representations of the first hitting time of a Bessel process. {\it Comm. Stoch. Anal\/}. {\bf 5} No. 4, pp. 701--719.
\bibitem{hh1} Hernandez-del-Valle, G. (2012). On the first time that a 3-D Bessel bridge hits a boundary. {\it Stoch. Models\/}, {\bf 28} No.4.
\bibitem{kar} Karatzas, I. and S. Shreve (1991) Brownian Motion and Stochastic Calculus, Springer-Verlag, New York.
\bibitem{kie} Kiefer, J. (1959) $K$-sample analogues of the Kolmogorov-Smirnov and Cram\'er-von Mises tests. {\it Ann. Math. Stat.\/}, {\bf 30}, pp. 420--447.
\bibitem{k} Kolmogorov, A. (1933) Sulla determinazione empirica di una legge di distribuzione. {\it G. Inst. Ita. Attuari\/}, {\bf 4}, p. 83.
\bibitem{pe} Peskir, G. (2001) On integral equations arising in the first-passage problem for Brownian motion. {\it J. Integral Equations Appl.\/}, {\bf 14}, pp. 397--423.
\bibitem{rp} Pinsky, R. G. (1995) Positive Harmonic Functions and Diffusion. Cambridge University Press, Cambridge.
\bibitem{py} Pitman, J. and M. Yor (1999) The law of the maximum of a Bessel bridge. {\it Electron. J. Probab.\/}, {\bf 4}, pp. 1--35.
\bibitem{tres} Revuz, D., and M. Yor. (2005). Continuous Martingales and Brownian Motion, Springer-Verlag, New York.
\bibitem{ros} Rosenbloom, P. D. and D. V. Widder (1959). Expansions in terms of heat polynomials and associated functions, {\it Transactions of the American Mathematical Society\/}, {\bf 92}, pp. 220--266.
\bibitem{s} Smirnov, N. V. (1948) Tables for estimating the goodness of fit of empirical distributions. {\it Ann. Math. Stat.\/}, {\bf 
19}, p. 279.
\bibitem{sy} Salminen, P. and M. Yor (2011) On hitting times of affine boundaries by reflecting Brownian motion and Bessel processes. {\it Periodica Math. Hungar.\/}, {\bf 62}, No. 1, pp. 75--101.
\end{thebibliography}
\end{document}